\documentclass[12pt,a4paper,english,reqno]{amsart}
\usepackage[a4paper,footskip=1.5em]{geometry}
\usepackage{amsmath,amssymb,amsthm,mathtools}
\usepackage[mathscr]{euscript}
\usepackage[usenames,dvipsnames]{color}
\usepackage{adjustbox,tikz,calc,graphics,babel,standalone}
\usepackage{subcaption}
\usepackage{csquotes,enumerate,verbatim}
\usepackage[final]{microtype}
\usepackage[numbers]{natbib}
\usetikzlibrary{shapes.misc,calc,intersections,patterns,decorations.pathreplacing}
\usepackage{hyperref}
\hypersetup{colorlinks=true,linkcolor=blue,citecolor=blue,pdfpagemode=UseNone,pdfstartview={XYZ null null 1.00}}
\usepackage{cmtiup}

\theoremstyle{plain}
\newtheorem*{theorem*}{Theorem}
\newtheorem{theorem}{Theorem}[section]
\newtheorem{lemma}[theorem]{Lemma}
\newtheorem{claim}[theorem]{Claim}
\newtheorem{proposition}[theorem]{Proposition}
\newtheorem*{claim*}{Claim}

\theoremstyle{remark}

\newcommand{\I}{\mathcal{I}}
\newcommand{\A}{\mathcal{A}}
\newcommand{\B}{\mathcal{B}}

\newcommand{\Prob}{\mathbb{P}}
\newcommand{\PN}{\mathcal{P}_n}
\newcommand{\N}{\mathbb{N}}

\newcommand\abs[1]{\left|#1\right|}

\let\emptyset\varnothing
\let\eps\varepsilon
\let\originalleft\left
\let\originalright\right
\renewcommand{\left}{\mathopen{}\mathclose\bgroup\originalleft}
\renewcommand{\right}{\aftergroup\egroup\originalright}

\begin{document}

\title{On regular 3-wise intersecting families}

\author{Keith Frankston}
\address{Department of Mathematics,	Rutgers University, Piscataway NJ 08854, USA}
\email{keith.frankston@math.rutgers.edu}

\author{Jeff Kahn}
\address{Department of Mathematics,	Rutgers University, Piscataway NJ 08854, USA}
\thanks{J.K. was supported by NSF Grant DMS1501962.}

\email{jkahn@math.rutgers.edu}

\author{Bhargav Narayanan}
\address{Department of Mathematics,	Rutgers University, Piscataway NJ 08854, USA}
\email{narayanan@math.rutgers.edu}

\date{22 November 2017}


\begin{abstract}
Ellis and the third author showed, verifying a conjecture of Frankl, that any $3$-wise intersecting family of subsets of $\{1,2,\dots,n\}$ admitting a transitive automorphism group has cardinality $o(2^n)$, while a construction of Frankl demonstrates that the same conclusion need not hold under the weaker constraint of being regular. Answering a question of Cameron, Frankl and Kantor from 1989, we show that the restriction of admitting a transitive automorphism group may be relaxed significantly: we prove that any $3$-wise intersecting family of subsets of $\{1,2,\dots,n\}$ that is regular and increasing has cardinality $o(2^n)$.
\end{abstract}
\maketitle

\section{Introduction}
This paper is primarily concerned with intersecting families: for an integer $r \geq 2$, a family of sets $\A$ is said to be \emph{$r$-wise intersecting} if any $r$ of the sets in $\A$ have nonempty intersection. There is by now a large body of work studying the extremal properties of families of sets under various intersection requirements; we refer the reader to the surveys~\citep{survey-1, survey-2} for an overview. A common theme that arises when studying the extremal properties of intersecting families is that the extremal constructions are often highly asymmetric; indeed, this is the case with many of the classical results in the field, such as the Erd\H{o}s--Ko--Rado theorem~\citep{EKR} and the Ahlswede--Khachatrian theorem~\citep{AK} to name just two. It is therefore natural to ask what, if anything, changes when one considers intersecting families subject to requirements of `symmetry', and this is the line of questioning that we pursue here.

For a positive integer $n \in \N$, let us write $[n]$ for the set $\{1, 2,\dots, n\}$, and $\PN$ for the power-set of $[n]$. We say that a family $\A \subseteq \PN$ is \emph{symmetric} if the automorphism group of $\A$ is transitive on $[n]$, \emph{regular} if every element of $[n]$ belongs to the same number of sets in $\A$, and \emph{increasing} if $\A$ is closed under taking supersets. We stress that the families we shall study here will be non-uniform, i.e., their members need not all be of the same size; for related work on uniform intersecting families, see the paper of Ellis, Kalai and the third author~\citep{elkal} addressing the symmetric case, and the results of Ihringer and Kupavskii~\citep{kup} addressing the regular case.

The family $\{x \subseteq [n]: |x| > n/2\}$ is a symmetric $2$-wise intersecting family containing a positive fraction of all the sets in $\PN$. Ellis and the third author~\citep{ellis}, verifying a conjecture of Frankl~\citep{frankl-1}, proved that symmetric $r$-wise intersecting families must be significantly smaller when $r \ge 3$; more precisely, they showed the following.

\begin{theorem}
\label{thm:sym}
If $\A \subseteq \PN$ is a symmetric $3$-wise intersecting family, then $|\A| = o(2^n)$.
\end{theorem}

On the other hand, a projective-geometric construction of Frankl~\citep{frankl-1} shows that there exist regular $3$-wise intersecting subfamilies of $\PN$ containing a positive fraction of all the sets in $\PN$, so the conclusion of Theorem~\ref{thm:sym} no longer holds when one considers regular families instead of symmetric ones. 

Here, we investigate the middle ground between symmetric and regular families following Cameron, Frankl and Kantor~\citep{frankl-2}: they proved that if $\A \subseteq \PN$ is a $4$-wise intersecting family that is both regular \emph{and increasing}, then $|\A| = o(2^n)$, and asked what one can say about regular $3$-wise intersecting families. Our main result answers this question by showing that the conclusion of Theorem~\ref{thm:sym} does hold for regular families, provided again that they are increasing.

\begin{theorem}
\label{thm:main}
If $\A \subseteq \PN$ is a $3$-wise intersecting family that is both regular and increasing, then $|\A| = o(2^n)$.
\end{theorem}

Of course, Theorem~\ref{thm:main} implies Theorem~\ref{thm:sym}; to see this, note that if $\A \subseteq \PN$ is a symmetric $3$-wise intersecting family, then $\{ y : x \subseteq y \text{ for some } x \in \A \}$ is a $3$-wise intersecting family containing $\A$ that is both regular and increasing. 

It is worth highlighting that in both~\citep{ellis} and the present work, Fourier analysis plays a crucial, if invisible, role: indeed, the proof of Theorem~\ref{thm:sym} hinges on a sharp threshold result of Friedgut and Kalai~\cite{fk}, while here, to prove the stronger assertion of Theorem~\ref{thm:main}, we in turn rely on the somewhat heavier machinery of Friedgut's junta theorem~\citep{junthm}. The main new technical tool that we develop to prove Theorem~\ref{thm:main} is a lemma demonstrating the existence of threshold-type behaviour under some rather mild conditions; this result (see Lemma~\ref{lem:st}) might be of some independent interest.

This paper is organised as follows. We collect the various tools we require in Section~\ref{sec:prelim}. The proof of Theorem~\ref{thm:main} follows in Section~\ref{sec:proof}. We conclude in Section~\ref{sec:conc} with a brief discussion of open problems.

\section{Preliminaries}\label{sec:prelim}
In this section, we briefly describe the notions and tools we shall require for our arguments.

For $0 \le p \le 1$, we write $\mu_p$ for the \emph{$p$-biased measure} on $\PN$, defined by
\[\mu_p(\{x\}) = p^{|x|}(1-p)^{n-|x|}\]
for all $x \subseteq [n]$. We abbreviate $\mu_{\frac{1}{2}}$ by $\mu$, and note that this is just the normalised counting measure.

For a family $\A \subseteq \PN$, we write $\I(\A) = \{x\cap y: x, y \in \A\}$ for the family of all possible intersections of pairs of sets from $\A$. We require the following proposition from~\citep{ellis}; we include a short proof for completeness.

\begin{proposition}
\label{prop:int}
For any $\A \subseteq \PN$, if $\mu_p(\A) \ge \delta$, then $\mu_{p^2}(\I(\A)) \ge \delta^2$.
\end{proposition}
\begin{proof}
Let $x$ and $y$ be two random elements of $\PN$ drawn independently according to the distribution $\mu_p$. It is then clear that $x \cap y$ has distribution $\mu_{p^2}$, so we have
\[\mu_{p^2}(\I(\A)) = \Prob(x \cap y \in \I(\A)) \ge \Prob(x,y \in \A) = \mu_p(\A)^2,\]
proving the proposition.
\end{proof}

We shall require the notions of influences and juntas. First, given $\A \subseteq \PN$, we say that an element $i \in [n]$ is \emph{pivotal} for $\A$ at $x \in \PN$ if exactly one of $x$ and $x \bigtriangleup \{i\}$ lies in $\A$, and for $0 \le p \le 1$, we define the \emph{total influence $I_p(\A)$} of $\A$ at $p$ to be the expected number of pivotal elements for $\A$ at a random set $x \in \PN$ drawn according to the distribution $\mu_{p}$. The following fundamental formula was originally observed independently by Margulis~\citep{margulis} and Russo~\citep{russo}.

\begin{proposition}
\label{prop:mr}
If $\A \subseteq \PN$ is increasing, then 
\[ \frac{d}{dp}\mu_p(\A) = I_p(\A)\]
for all $0 < p < 1$.\qed
\end{proposition}

Next, for $J \subseteq [n]$, a family $\A \subseteq \PN$ is said to be a \emph{$J$-junta} if the membership of a set in $\A$ is determined by its intersection with $J$, or in other words, if $x \in \A$ and $x \cap J = y \cap J$ for some $y \in \PN$, then this implies that $y \in \A$. The following result due to Friedgut~\citep{junthm} will be our main tool.

\begin{theorem}
\label{prop:fk}
For each $C> 0$ and $0 <\eps< 1$, there exists $K>0$ such that the following holds for all $\eps \le p \le 1-\eps$ and $n \in \N$. For any $\A \subseteq \PN$ with $I_p(\A) \le C$, there exists a set $J\subseteq [n]$ with $|J| \le K$ and a $J$-junta $\B \subseteq \PN$ such that $\mu_p(\A \bigtriangleup \B) \le \eps$. \qed
\end{theorem}

Finally, we say that two families $\A,\B \subseteq \PN$ are \emph{cross-intersecting} if $x \cap y \neq \emptyset$ for all $x\in \A$ and $y \in \B$. We need the following simple fact also used in~\citep{ellis}.
\begin{proposition}
\label{prop:cross}
If $\A,\B \subseteq \PN$ are cross-intersecting, then
\[\mu_p(\A) + \mu_{1-p}(\B) \le 1\]
for any $0 \le p \le 1$.
\end{proposition}
\begin{proof}
Since $\A$ and $\B$ are cross-intersecting, it is clear that $\A \subseteq \PN \setminus \tilde{\B}$, where $\tilde{\B} = \{[n] \setminus x : x \in \B\}$. Therefore,
\[\mu_p(\A) \le \mu_p(\PN \setminus \tilde{\B}) = 1 - \mu_{p}(\tilde{\B}) = 1-\mu_{1-p}(\B). \qedhere\]
\end{proof}

\section{Proof of the main result}\label{sec:proof}
Our proof of Theorem~\ref{thm:main} borrows ideas from both~\citep{frankl-2} and \citep{ellis}. Before turning to the proof, let us briefly explain what is lost, relative to the argument in~\citep{ellis}, by dropping the requirement of symmetry: for a family $\A \subseteq \PN$ that is both symmetric and increasing, a result of Talagrand~\citep{tal} guarantees that the total influence $I_p(\A)$ is large whenever $\mu_p(\A)$ is bounded away from both $0$ and $1$, which ensures, by Proposition~\ref{prop:mr}, that the derivative of $\mu_p(\A)$ with respect to $p$ is also large under these circumstances; this is no longer the case when one considers regular families as opposed to symmetric ones. A replacement for this fact, the main new ingredient here, is the following lemma asserting a somewhat weaker version of this threshold behaviour under milder conditions.

\begin{lemma}\label{lem:st}
For any $\eps, \delta>0$, the following holds for all sufficiently large $n \in \N$. If $\A \subseteq \PN$ is both regular and increasing, and $\mu (\A) \ge \delta$, then $\mu_{\frac{1}{2}+\eps} (\A) \ge 1-\eps$.
\end{lemma}
\begin{proof}
In what follows, we fix $\eta = \eps\delta/(2+\delta)$ and additionally suppose that $n$ is large enough for all our estimates to hold; in particular, constants suppressed by the asymptotic notation may depend on $\eps$ and $\delta$ but, of course, not on $n$. 

Since $\mu(\A) = \mu_{\frac{1}{2}} (\A) \ge \delta$ and $\mu_{\frac{1}{2}+\eps} (\A) \le 1$, it follows from Proposition~\ref{prop:mr} that there exists $q \in [1/2, 1/2+\eps]$ such that $I_q(\A) \leq 1/\eps$. Theorem~\ref{prop:fk} now implies that there exists $J \subseteq [n]$ with $|J|= K$ and a $J$-junta $\B \subseteq \PN$ such that $\mu_q(\A\bigtriangleup\B) \le \eta$, where $K$ is a constant depending only on $\eps$ and $\delta$. 

Let us set up some notation before we proceed. For $i\in[n]$, let $\A_i$ denote the family of those sets in $\A$ containing $i$, and for $y \subseteq J$, define the \emph{fibre $\A(y)$} of $\A$ over $y$ by 
\[\A(y)=\lbrace x\setminus y : x\in\A \text{ and } x\cap J=y\rbrace.\] Also, let $\B'$ be the family on $J$ determining $\B$, i.e., $x \in \PN$ belongs to $\B$ if and only if $x \cap J$ belongs to $\B'$.

We first note that as $\A$ is regular, the sets $\A_i$ are all roughly half as large as $\A$; a similar observation is used in~\citep{frankl-2}.
\begin{claim}\label{claim:ent}
For each $i \in [n]$, we have $\mu (\A)/2 \le \mu(\A_i) \le \mu (\A)/2 +O(1/\sqrt{n})$.
\end{claim}
\begin{proof}
The first inequality follows from the fact that $\A$ is increasing, so it suffices to verify the second. Let $Z$ be a set drawn uniformly at random from $\A$, and for $i \in [n]$, let $Z_i$ be the indicator of the event $\{i \in Z\}$. We shall rely on the properties of the binary entropy $H(\cdot)$ of a random variable; see~\citep{book} for the basic notions. It follows from the sub-additivity of entropy that $H(Z) \le \sum_{i =1}^{n} H(Z_i)$. Clearly, we have \[H(Z) = \log_2 |\A| = n + \log_2 (\mu (\A))\ge n + \log_2 \delta,
\] and, writing $\vartheta$ for the common value of $|\A_i|/|\A|$ for all $i \in [n]$, we also have \[H(Z_i) = -\vartheta\log_2 \vartheta - (1 - \vartheta) \log_2 (1 - \vartheta)\] for each $i \in [n]$. It is now easy to verify from the sub-additivity estimate above that $\vartheta = 1/2 + O(1/\sqrt{n})$, proving the claim.
\end{proof}

Next, we observe that all the fibres of $\A$ have roughly the same size as well. Let us write $\sigma_p$ for the $p$-biased measure on the power set of $J$ and $\tau_p$ for the $p$-biased measure on the power set of $[n] \setminus J$, so that $\mu_p = \sigma_p \times \tau_p$, and again, we abbreviate $\sigma_{\frac12}$ and $\tau_{\frac12}$ by $\sigma$ and $\tau$ respectively.
\begin{claim}
\label{claim:fibre}
For all $y\subseteq J$, we have $\tau(\A(y)) = \mu (\A) + o(1)$.
\end{claim}
\begin{proof}
We note that
\[
\mu(\A) = \sum_{y \subseteq J} \sigma(y) \tau(\A(y)),
\]
and that $\sigma(y) = 2^{-K}$ for all $y \subseteq J$. For any $i \in y \subseteq J$, we have $\A(y\setminus \{ i \}) \subseteq \A(y)$ because $\A$ is increasing, so \[\tau(\A(y)) \ge \tau(\A(y\setminus \{ i \})).\]
Since $|J| = K = O(1)$, to prove the claim, it clearly suffices to show that for any $i \in y \subseteq J$, we have
\[\tau(\A(y)) \le \tau(\A(y\setminus \{ i \}))+ o(1);\]
indeed, this would imply that 
\[\tau(\A(y))=\tau(\A({\emptyset}))+ o(\abs{y})=\tau(\A({\emptyset}))+o(1)\]
for each $y \subseteq J$, and the claim would follow.

Fix $i \in J$, and note that
\[
\mu(\A_i) = \sum_{i\in y \subseteq J} \sigma(y) \tau(\A(y)),
\]
so we have
\[\mu(\A_i) - \mu(\A)/2 = 2^{-K-1} \sum_{i\in y \subseteq J} \left(\tau(\A(y)) - \tau(\A(y\setminus \{ i \}))\right).\]
We know from Claim~\ref{claim:ent} that $\mu(\A_i) - \mu(\A)/2 = O(1/\sqrt{n})$, so for each $y \subseteq J$ containing $i$, we have 
\[\tau(\A(y)) - \tau(\A(y\setminus \{ i \})) = O(1/\sqrt{n}),\]
as required.
\end{proof}

We may now complete the proof of the lemma. Recall that we earlier fixed $q \in [1/2, 1/2+\eps]$ and a $J$-junta $\B \subseteq \PN$ such that $\mu_q(\A\bigtriangleup\B) \le \eta$, and defined $\B'$ to be the family on $J$ determining $\B$. 

First, note that 
\[\mu_q(\A\bigtriangleup\B) = \sum_{y\in \B'} \sigma_q(y)\left(1-\tau_q(\A(y))\right) + \sum_{y \not\in \B'}\sigma_q(y)\left(\tau_q(\A(y))\right).\]
Since $\A$ is increasing, we see from Claim~\ref{claim:fibre} that $\tau_q(\A(y)) \ge \tau(\A(y)) \ge \delta/2$ for all $y \subseteq J$. Therefore, since $\mu_q(\A\bigtriangleup\B)  \le \eta$, we see that
\[ \sum_{y \not\in \B'}\sigma_q(y) \le 2\eta/\delta,\]
which implies that
\[ \mu_q(\B)  = \sum_{y \in \B'} \sigma_q(y) \ge 1 - 2\eta/\delta.\]
Again, since $\mu_q(\A\bigtriangleup\B)  \le \eta$ and $\eta=\eps\delta/(2+\delta)$, it follows that 
\[ \mu_{\frac{1}{2}+\eps}(\A) \ge \mu_q(\A) \ge 1 - 2\eta/\delta - \eta = 1 - \eps,
\]
proving the lemma.
\end{proof}

Armed with Lemma~\ref{lem:st}, we may now prove Theorem~\ref{thm:main}; the proof below by and large follows the argument in~\citep{ellis}, with Lemma~\ref{lem:st} serving as a substitute for the sharp threshold result used there.

\begin{proof}[Proof of Theorem~\ref{thm:main}]
We need to show for any fixed $\delta >0$, that for all but finitely many $n \in \N$, if $\A \subseteq \PN$ is a $3$-wise intersecting family that is both regular and increasing, then $\mu(\A) < \delta$; hence, suppose for a contradiction that $n$ is sufficiently large and that $\A \subseteq \PN$ is a family as just described with $\mu(\A) \ge \delta$. 

Let us fix $\eps = \min\{1/4, \delta^2/2\}$. First, since $\A$ is increasing, we know from Lemma~\ref{lem:st} that \[\mu_{\frac{3}{4}} (\A)\ge \mu_{\frac{1}{2}+\eps}(\A) \ge 1- \eps > 1-\delta^2.\]
Next, by Proposition~\ref{prop:int}, we have \[\mu_{\frac14} (\I(\A)) \geq \delta^2.\] 
Finally, since $\A$ is a $3$-wise intersecting family, $\A$ and $\I(\A)$ are cross-intersecting, so we conclude from Proposition~\ref{prop:cross} that 
\[\mu_{\frac34}(\A)\leq 1 - \mu_{\frac14}(\I(\A)) \leq 1- \delta^2,\]
yielding a contradiction, and establishing the result.
\end{proof}

\section{Conclusion}\label{sec:conc}
The best bound for Theorem~\ref{thm:main} that we may read out of the argument here is rather poor on account of our reliance on the junta theorem; it would therefore be interesting to improve this. Concretely, it would be good to decide if any $3$-wise intersecting family $\A \subseteq \PN$ that is both regular and increasing must satisfy 
\[\log_2 |\A| \le n - cn^{\delta},\] 
where $c,\delta>0$ are universal constants; as evidenced by the constructions in~\citep{ellis}, a bound of this type would be the best one could hope for. We  ought to point out that we do not yet know how to prove an estimate of the above form even for \emph{symmetric} $3$-wise intersecting families; what is known however is that such an estimate does hold for symmetric $4$-wise intersecting families, as was shown by Cameron, Frankl and Kantor~\citep{frankl-2}.

\bibliographystyle{amsplain}
\bibliography{regular_3_families}

\end{document}